\documentclass[12pt]{article}

\usepackage{amsfonts}
\usepackage{amsmath}
\usepackage{amssymb}
\usepackage{color}
\usepackage{amsthm}
\usepackage{hyperref}

\newtheorem{theorem}{Theorem}

\newtheorem{lemma}[theorem]{Lemma}

\newtheorem{remark}[theorem]{Remark}

\numberwithin{theorem}{section}
\numberwithin{equation}{section}

\title{The Borel Distribution: Approximation and Concentration} \author{Fraser Daly\footnote{Department of Actuarial Mathematics and Statistics, Heriot-Watt University, Edinburgh EH14 4AS, UK.  E-mail: f.daly@hw.ac.uk}\, 
and Seva Shneer\footnote{Department of Actuarial Mathematics and Statistics, Heriot-Watt University, Edinburgh EH14 4AS, UK.  E-mail: v.shneer@hw.ac.uk}} \date{\today}

\begin{document}

\maketitle

\noindent{\bf Abstract} 
We develop the tools necessary to use Stein's method for approximation by a Borel distribution, which we illustrate by considering the approximation of the number of customers served in the busy period of an M/G/1 queue. We further derive concentration inequalities for the Borel distribution. Both these sets of results are based on a representation for the size-biased version of a Borel random variable.
\vspace{12pt}


\vspace{12pt}


\section{Introduction}

The Borel distribution is well-known in the setting of Galton--Watson processes with a Poisson offspring distribution. In such a process, the total progeny $Z$ satisfies
\begin{equation}\label{eq:borelrep}
Z\stackrel{d}{=}1+\sum_{i=1}^\xi Z_i\,,
\end{equation}
where `$\stackrel{d}{=}$' denotes equality in distribution, $Z_1,Z_2,\ldots$ are IID copies of $Z$, and $\xi\sim\mbox{Po}(\lambda)$ has a Poisson distribution with mean $\lambda<1$, independent of the $Z_i$. This random variable $Z$ is said to have a Borel distribution with parameter $\lambda$, written $Z\sim\mbox{Borel}(\lambda)$. The corresponding mass function is
\[
\mathbb{P}(Z=j)=\frac{e^{-\lambda j}(\lambda j)^{j-1}}{j!}\,,
\] 
for $j\in\mathbb{N}=\{1,2,\ldots\}$. We note from \cite{haight60} that $\mathbb{E}[Z]=(1-\lambda)^{-1}$ and $\mbox{Var}(Z)=\lambda(1-\lambda)^{-3}$.

Given the relationship between branching processes and queueing systems, in which customers arriving during a particular service period can be thought of as the offspring of the customer currently in service (see page 284 of \cite{asmussen03} for details), the relationship (\ref{eq:borelrep}) shows that we may also think of the Borel distribution as describing the number of customers served in a busy period of an M/D/1 queue, where arrivals occur at the points of a Poisson process of rate $\lambda<1$ and service times are deterministic (of length 1). It is thus natural to consider the Borel distribution as an approximation for the number of customers served in a busy period of an M/G/1 queue, with a general distribution of service times (with mean 1) which are `not too far from' deterministic.

It is also natural to consider the Borel distribution as a limit or approximation in some settings where structure analogous to (\ref{eq:borelrep}) appears. For example, consider an Erd\H{o}s--R\'enyi random graph with $n$ vertices, where each pair of vertices is connected by an edge independently with probability $\lambda/n$, for some constant $\lambda<1$. In the large-$n$ limit, the locally tree-like behaviour of this random graph is well-known, and as a consequence gives a limiting Borel distribution for the size of a typical connected component. See, for example, \cite{janson08} for a detailed discussion of this, and \cite{vonbahr80,ball95} for  a discussion of the closely related Borel limit for the total number of infected individuals in a Reed--Frost epidemic model, including some explicit error bounds in the corresponding approximation. 

In this note we will explore properties of the Borel distribution, with two objectives in mind. The first is to develop the tools necessary in order to apply Stein's method for probability approximation to problems in which the Borel distribution is a natural limiting or approximating object. Stein's method is a powerful technique in a wide variety of probabilistic approximation settings; see \cite{ross11} and references therein for a general introduction. Our second objective is to establish concentration inequalities for the Borel distribution.   

Our starting point for both these directions is a representation for the size-biased version of $Z\sim\mbox{Borel}(\lambda)$; see the definition (\ref{eq:sb}) of size biasing and the representation (\ref{eq:borelsizebias}) below. Size biasing has previously been used to investigate the Borel distribution by Aldous and Pitman \cite{aldous98}, who used a representation of the mass function of $Z$ in terms of its size-biased version to give elegant proofs of various properties of $Z$. They also derive and apply an elegant representation of the size-biased version of $Z$ as a geometric sum (see Remark \ref{rem:aldous} below). Our work is complementary to \cite{aldous98}, using a different representation of the size-biased version of $Z$ and with different aims.

We will develop Stein's method for approximation by the Borel distribution in Section \ref{sec:stein}, which we illustrate with an application to the approximation of the number of customers served in the busy period of an M/G/1 queue in Section \ref{sec:q}. Our work here leaves open several questions for future research. For example, our results allow us to consider only a limited range of values of $\lambda$, rather than all $\lambda<1$. This seems to be a result of the choice of Stein equation we make here (see Section \ref{sec:stein} below for further details), but it is a restriction we have been unable to remove and it remains an open problem to derive results applicable for all $\lambda<1$. Finally, in Section \ref{sec:concentration} we derive concentration inequalities for the Borel distribution.

\section{Stein's method for Borel approximation}\label{sec:stein}

In this section we establish a general bound for Borel approximation using Stein's method; our main result is Theorem \ref{thm:stein} below. This is the first use of Stein's method for Borel approximation. We refer the reader to \cite{ross11} and references therein for background on Stein's method, a powerful technique for deriving error bounds in distributional approximations in a wide variety of settings. An application of our results is given in Section \ref{sec:q}.

Throughout this section we let $Z\sim\mbox{Borel}(\lambda)$ and let $W$ be any positive integer-valued random variable with $\mathbb{E}[W]=(1-\lambda)^{-1}$ for some $0<\lambda<1$. For any such $W$, we denote by $W^\star$ the size-biased version of $W$, defined by
\begin{equation}\label{eq:sb}
\mathbb{P}(W^\star=j)=\frac{j\mathbb{P}(W=j)}{\mathbb{E}W}\,.
\end{equation}
Note that $W^\star$ satisfies $\mathbb{E}[f(W^\star)]=\frac{\mathbb{E}[Wf(W)]}{\mathbb{E}W}$ for any $f:\mathbb{N}\to\mathbb{R}$ for which these expectations exist; this latter formula is used to define the size-biased version of non-negative random variables which are not integer-valued.

We use the representation (\ref{eq:borelrep}) to write a distributional equality for the size-biased version $Z^\star$, which we can then use to define an appropriate Stein equation to use as the basis of Stein's method for Borel approximation. This distributional equality will also be the starting point for the concentration inequalities we derive in Section \ref{sec:concentration}.

Let $I$ be a Bernoulli random variable, independent of all else, with expectation $\lambda$. Noting that $\mathbb{E}Z=(1-\lambda)^{-1}$, results from Sections 2.2 and 2.4 of \cite{arratia19} on size biasing sums and random sums of random variables give us that
\begin{equation} \label{eq:borelsizebias}
Z^\star\stackrel{d}{=}(1-I)Z+I(Z+Z^\star)\,,
\end{equation}
where $Z$ and $Z^\star$ on the right-hand side are independent. To see this, note that when size biasing the sum of the two terms on the right-hand side of (\ref{eq:borelrep}), we choose one of these two terms (each with probability proportional to its mean) and replace it with a size-biased version; the random variable $I$ is an indicator that we chose to size bias the term $\sum_{i=1}^\xi Z_i$ here. The size-biased version of this latter sum may be written as $Z^\star+\sum_{i=1}^{\xi^\star-1}Z_i$ by results from \cite{arratia19}. Finally, we note that for a Poisson random variable $\xi$ we have that $\xi^\star-1$ is equal in distribution to $\xi$. See Sections 2.2 and 2.4 of \cite{arratia19} for further details and formal statements of these rules for size biasing.

\begin{remark}\label{rem:aldous}
\emph{An alternative representation of $Z^\star$ was given by Aldous and Pitman \cite{aldous98}, who showed in their Corollary 32 that $Z^\star\stackrel{d}{=}Z_1+\cdots+Z_\eta$, where $Z_1,Z_2,\ldots$ are independent copies of $Z$, and $\eta$ has a geometric distribution (independent of the $Z_i$) with mass function $\mathbb{P}(\eta=j)=(1-\lambda)\lambda^{j-1}$ for $j=1,2,\ldots$. For many of our purposes here, the representation (\ref{eq:borelsizebias}) is more directly useful than this one. We comment a little further on this below.}
\end{remark}

The representation (\ref{eq:borelsizebias}) motivates us to define the following Stein equation to compare $W$ and $Z$: for each bounded $h:\mathbb{N}\to\mathbb{R}$, we let $f_h:\mathbb{N}\to\mathbb{R}$ be the solution of the equation
\begin{equation} \label{eq:borelstein}
h(k)-\mathbb{E}[h(Z)]=(1-\lambda)(k-1)f_h(k)-\lambda(1-\lambda)k\sum_{i=1}^\infty f_h(i+k)q(i)\,,
\end{equation}
where $q(i)=\mathbb{P}(Z=i)$, and we define $f_h(1)=0$ for all $h$. Such an equation is the starting point for applying Stein's method to derive error bounds in the approximation of $W$ by $Z$. We will first see how such error bounds follow from (\ref{eq:borelstein}), before using the remainder of this section to bound the solution $f_h$ to (\ref{eq:borelstein}). Using (\ref{eq:borelstein}), we have that
\begin{align}\label{eq:borelcalc}
\nonumber\mathbb{E}[h(W)]-\mathbb{E}[h(Z)]
&=(1-\lambda)\mathbb{E}[(W-1)f_h(W)]-\lambda(1-\lambda)\mathbb{E}\left[W\sum_{i=1}^\infty f_h(i+W)q(i)\right]\\
\nonumber&=\mathbb{E}[f_h(W^\star)]-(1-\lambda)\mathbb{E}[f_h(W)]-\lambda\mathbb{E}[f_h(Z+W^\star)] \\ &= \mathbb{E}[f_h(W^\star)]-\mathbb{E}[f_h((1-I)W+I(Z+W^\star))]\,,
\end{align}
where $Z$ and $W^\star$ are independent. Our error bounds will be in terms of the total variation distance, defined by
\begin{multline*}
d_{TV}(\mathcal{L}(W),\mathcal{L}(Z))=\sup_{A\subseteq\mathbb{N}}|\mathbb{P}(W\in A)-\mathbb{P}(Z\in A)|\\
=\frac{1}{2}\sup_{\substack{h:\mathbb{N}\to\mathbb{R}\\ \lVert h\rVert\leq1}}|\mathbb{E}h(W)-\mathbb{E}h(Z)|
=\inf_{(W,Z)}\mathbb{P}(W\not=Z)\,,
\end{multline*}
where $\lVert\cdot\rVert$ is the supremum norm, and the infimum is taken over all couplings of $(W,Z)$. Hence, from (\ref{eq:borelcalc}), 
\begin{align}
\nonumber d_{TV}(\mathcal{L}(W),\mbox{Borel}(\lambda))&=\frac{1}{2}\sup_{\lVert h\rVert\leq1}|\mathbb{E}[f_h(W^\star)]-\mathbb{E}[f_h((1-I)W+I(Z+W^\star))]|\\
\label{eq:borelTV}&\leq\sup_{\|h\|\leq1}\|f_h\|d_{TV}(\mathcal{L}(W^\star),\mathcal{L}((1-I)W+I(Z+W^\star)))\,.
\end{align}
From the distributional equality (\ref{eq:borelsizebias}), it is clear that this upper bound is zero if $W$ and $Z$ are equal in distribution.

We now establish that the solution $f_h$ to the equation (\ref{eq:borelstein}) is bounded. Recall that we have defined $f_h(1)=0$. For $k\geq2$, we are motivated by the representation used by \cite{bcl92} (see Equation (6) therein) for the solution of a similar Stein equation for compound Poisson approximation to write 
\begin{equation}\label{eq:steinrep}
f_h(k)=\sum_{m=k}^\infty a_{k,m}h_Z(m)\,, 
\end{equation}
for some coefficients $a_{k,m}$ which we will bound below, where $h_Z(k)=(1-\lambda)^{-1}(h(k)-\mathbb{E}[h(Z)])$, which satisfies
\begin{equation} \label{eq:borelstein2}
h_Z(k)=(k-1)f_h(k)-\lambda k\sum_{i=1}^\infty f_h(i+k)q(i)\,.
\end{equation}
Since we have $\lVert h_Z\rVert\leq(1-\lambda)^{-1}$ for $h$ with $\lVert h\rVert\leq1$, it follows that, for each such $h$ and $k\geq2$,
\begin{equation}\label{eq:coeff}
|f_h(k)|\leq\frac{1}{1-\lambda}\sum_{m=k}^\infty|a_{k,m}|\,.
\end{equation}
We therefore bound $f_h$ by first bounding the coefficients $a_{k,m}$.
\begin{lemma}\label{lem:bound}
For $k\geq2$, $a_{kk}=(k-1)^{-1}$ and $|a_{k,k+j}|\leq\frac{j\lambda q(j)}{k-1}$ for $j \ge 1$.
\end{lemma}
\begin{proof}
Substituting the representation (\ref{eq:steinrep}) into (\ref{eq:borelstein2}), we have that
\begin{equation} \label{eq:lem21}
h_Z(k)=(k-1)\sum_{m=k}^\infty a_{k,m}h_Z(m)-\lambda k\sum_{m=k+1}^\infty h_Z(m)\sum_{i=1}^{m-k}a_{k+i,m}q(i)\,,
\end{equation}
for each $k\geq2$. Comparing coefficients of $h_Z(k)$ on each side of (\ref{eq:lem21}) we have that $a_{k,k}=(k-1)^{-1}$, as required. Comparing coefficients of $h_Z(k+j)$ on each side of (\ref{eq:lem21}) for each $j\geq1$, we obtain
\begin{equation} \label{eq:lem22}
a_{k,k+j}=\frac{k\lambda}{k-1}\sum_{i=1}^ja_{k+i,k+j}q(i)\,,
\end{equation}
from which it follows that $a_{k,k+1}=\frac{\lambda q(1)}{k-1}$. This gives the required result in the case $j=1$, which acts as the base case in a proof by induction on $j$. To that end, we assume that $a_{k,k+l}\leq\frac{l\lambda q(l)}{k-1}$ for each $l<j$, and proceed by noting that (\ref{eq:lem22}) combined with this inductive assumption gives us
\begin{align*}
|a_{k,k+j}|&\leq\frac{k\lambda}{k-1}\left(\frac{q(j)}{k+j-1}+\sum_{i=1}^{j-1}\frac{(j-i)\lambda q(j-i)q(i)}{k+i-1}\right)\\
&\leq\frac{\lambda}{k-1}\left(q(j)+\sum_{i=1}^{j-1}e^{-\lambda j}\lambda^{j-1}\frac{i^{i-1}(j-i)^{j-i}}{(j-i)!i!}\right)\\
&=\frac{\lambda q(j)}{k-1}\left(1+\sum_{i=1}^j\binom{j}{i}\frac{i^{i-1}(j-i)^{j-i}}{j^{j-1}}\right)
=\frac{j\lambda q(j)}{k-1}\,,
\end{align*}
where the final equality follows from Abel's generalization of the binomial theorem (see Section 1.5 of \cite{riordan68}).
\end{proof}
\begin{lemma}\label{lem:steinfactor}
For each $\lVert h\rVert\leq1$, $\lVert f_h\rVert\leq(1-\lambda)^{-2}$.
\end{lemma}
\begin{proof}
Combining (\ref{eq:coeff}) with Lemma \ref{lem:bound}, we have that, for each $h$ with $\lVert h\rVert\leq1$ and each $k\geq2$,
\[
|f_h(k)|\leq\frac{1}{(1-\lambda)(k-1)}\left(1+\lambda\sum_{j=1}^\infty jq(j)\right)=\frac{1}{(1-\lambda)^2(k-1)}\leq\frac{1}{(1-\lambda)^2}\,,
\]
as required.
\end{proof}

Combining (\ref{eq:borelTV}) with Lemma \ref{lem:steinfactor}, we have proved the following.
\begin{theorem}\label{thm:stein}
Let $W$ be a positive integer-valued random variable with $\mathbb{E}W=(1-\lambda)^{-1}$ for some $0<\lambda<1$, and $W^\star$ denote its size-biased version. Then
\[
d_{TV}(\mathcal{L}(W),\mbox{Borel}(\lambda))\leq \frac{1}{(1-\lambda)^2} d_{TV}(\mathcal{L}(W^\star),\mathcal{L}((1-I)W+I(Z+W^\star)))\,,
\]
where $I$ is a Bernoulli random variable with $\mathbb{P}(I=1)=\lambda$, $Z\sim\mbox{Borel}(\lambda)$, and all random variables are independent.
\end{theorem}
\begin{remark}\label{rem:stein}
\emph{One unusual feature of the bound given in Theorem \ref{thm:stein}, compared to results using Stein's method in other settings, is the presence of the random variable $Z$ in the upper bound. In the application of Section \ref{sec:q}, this leads directly to a restriction on the values of $\lambda$ we may consider. We have been unable to find an alternative approach to Stein's method for Borel approximation which both yields a Stein equation with a tractable solution and removes this dependence on $Z$ in the upper bound. The complexity of expressions such as $\mathbb{P}(Z=j+1)/\mathbb{P}(Z=j)$ makes it difficult to apply the well-known density approach to deriving Stein equations, or its generalizations such as those detailed in \cite{ley17}. Similarly, the probability generating function  of $Z$ (see \cite{haight60}) does not give a recurrence relation of the form needed to apply the techniques of \cite{upadhye17} for deriving Stein operators. Yet another approach which we have considered is to use the representation  of Remark \ref{rem:aldous} in place of (\ref{eq:borelsizebias}) as the basis for a Stein equation.}

\emph{One way to ease the restrictions on $\lambda$, while employing the same Stein equation, would be to make use of a better bound on $f_h$ than that given by Lemma \ref{lem:steinfactor}. Although this seems to give the best possible uniform bound, its proof makes it clear that it is a poor upper bound for large $k$. Using a tighter upper bound on $|f_h(j)-f_h(k)|$, allowed to depend on $j$ and $k$, in establishing a result analogous to Theorem \ref{thm:stein} would be likely to allow a somewhat increased range of values of $\lambda$ to be considered in applications, at the cost of additional complexity in the analysis.}
\end{remark}

\section{Number of customers served in the busy period of an M/G/1 queue}\label{sec:q}

Consider an M/G/1 queue, with arrivals at the points of a Poisson process of rate $\lambda<1$, and IID service times distributed as the random variable $S$ with $\mathbb{E}S=1$. The condition $\lambda<1$ is necessary and sufficient for stability of the system. Let $N$ be the number of customers served during a busy period of the queue. That is, if a single customer arrives to an otherwise empty system, $N$ represents the number of customers served (including this original customer) before the system is again empty. Let $\nu$ be the number of customers arriving during the service of the customer initiating the busy period. Note that $\nu$ has the mixed Poisson distribution $\nu\sim\mbox{Po}(\lambda S)$. That is, $\nu|S\sim\mbox{Po}(\lambda S)$ with probability 1. Our starting point in Borel approximation for $N$ is the well-known representation
\begin{equation} \label{eq:Nrep}
N\stackrel{d}{=}1+\sum_{i=1}^\nu N_i\,,
\end{equation}
where $N_1,N_2,\ldots$ are independent copies of $N$ which are also independent of $\nu$. In the special case of the M/D/1 queue, where all service times are constant, this reduces to (\ref{eq:borelrep}), showing that the number of customers served in a busy period of an M/D/1 queue has a Borel distribution. Note also that (\ref{eq:Nrep}) implies that $\mathbb{E}[N]=(1-\lambda)^{-1}$.

We will consider the approximation of $N$ by $Z\sim\mbox{Borel}(\lambda)$ to illustrate the application and limitation of the results of Section \ref{sec:stein}. We begin by outlining a more direct argument for bounding the total variation distance between $N$ and $Z$, before then moving on to show an alternative argument based on a simple application of our Theorem \ref{thm:stein}. As we will see, this latter argument yields an error bound of the expected order, but only for a limited range of $\lambda$. This is for the reasons noted in Remark \ref{rem:stein}; overcoming the obstacles noted there is an open problem.

We could use the representations (\ref{eq:borelrep}) and (\ref{eq:Nrep}) to write
\begin{align*}
d_{TV}(\mathcal{L}(N),\mbox{Borel}(\lambda))&=d_{TV}\left(\mathcal{L}\left(\sum_{i=1}^\nu N_i\right),\mathcal{L}\left(\sum_{i=1}^\xi Z_i\right)\right)\\
&\leq d_{TV}(\mbox{Po}(\lambda S),\mbox{Po}(\lambda))+\lambda d_{TV}(\mathcal{L}(N),\mbox{Borel}(\lambda))\,,
\end{align*}
where the inequality uses Corollary 3.1 of \cite{vellaisamy96}. Using Theorem 1.C of \cite{barbour92}, for all $\lambda<1$ we thus have 
\begin{equation}\label{eq:qbd1}
d_{TV}(\mathcal{L}(N),\mbox{Borel}(\lambda))\leq\frac{1}{1-\lambda}d_{TV}(\mbox{Po}(\lambda S),\mbox{Po}(\lambda))\leq\frac{\lambda^2\mbox{Var}(S)}{1-\lambda}\,.
\end{equation}

Alternatively, we can use an approach based on our Theorem \ref{thm:stein}. Similarly to (\ref{eq:borelsizebias}), we use (\ref{eq:Nrep}) to write
\[
N^\star\stackrel{d}{=}(1-I)N+I\left(1+\sum_{i=1}^{\nu^\star-1}N_i+N^\star\right)\,,
\]
where, as before, $I$ is a Bernoulli random variable, independent of all else, such that $\mathbb{P}(I=1)=\lambda$. Then, using Theorem \ref{thm:stein}, conditioning on $I$ and using the representation (\ref{eq:borelrep}) we obtain
\[
d_{TV}(\mathcal{L}(N),\mbox{Borel}(\lambda))\leq\frac{\lambda}{(1-\lambda)^2}d_{TV}\left(\mathcal{L}\left(\sum_{i=1}^{\nu^\star-1}N_i\right),\mathcal{L}\left(\sum_{i=1}^\xi Z_i\right)\right)\,.
\]
Proceeding as we did above by using Corollary 3.1 of \cite{vellaisamy96}, we then have that
\[
d_{TV}(\mathcal{L}(N),\mbox{Borel}(\lambda))\leq\frac{\lambda}{(1-\lambda)^2}\left[d_{TV}(\mathcal{L}(\nu^\star-1),\mbox{Po}(\lambda))+\lambda d_{TV}(\mathcal{L}(N),\mbox{Borel}(\lambda))\right].
\]
This time there is a restriction on the values of $\lambda$ we allow. Gathering both instances of $d_{TV}(\mathcal{L}(N),\mbox{Borel}(\lambda))$ together, we require $\lambda<1/2$ to obtain a meaningful upper bound. Under this assumption,
\begin{equation}\label{eq:qbd2}
d_{TV}(\mathcal{L}(N),\mbox{Borel}(\lambda))\leq\frac{\lambda}{1-2\lambda}d_{TV}(\mathcal{L}(\nu^\star-1),\mbox{Po}(\lambda))
\leq\frac{\lambda^2\mathbb{E}|S^\star-1|}{1-2\lambda}=\frac{\lambda^2\mathbb{E}\left[S|S-1|\right]}{1-2\lambda}\,,
\end{equation}
where the final inequality follows from Theorem 1.C of \cite{barbour92} since $\nu^\star-1\sim\mbox{Po}(\lambda S^\star)$; see Section 2.2 of \cite{arratia19}. 

Each of the bounds (\ref{eq:qbd1}) and (\ref{eq:qbd2}) are of order $O(\lambda^2)$ in $\lambda$. The main difference between them is in the permitted values of $\lambda$; (\ref{eq:qbd1}) holds for all $\lambda<1$, while (\ref{eq:qbd2}) needs $\lambda<1/2$. The dependence on $S$ is also different between the two bounds. If $S$ is supported on the non-negative integers, then $\mathbb{E}|S^\star-1|=\mbox{Var}(S)$, though in general these quantities are different, with $\mbox{Var}(S)$ a little smaller than $\mathbb{E}|S^\star-1|$ in a few examples for which we computed them.

\section{Concentration bounds for the Borel distribution}\label{sec:concentration}

In this section we will derive concentration inequalities for $Z\sim\mbox{Borel}(\lambda)$, motivated by the complexity of the mass function of $Z$, which makes the distribution function of $Z$ difficult to evaluate explicitly. 

It is well-known that the Borel distribution is infinitely divisible, and so Theorem 5.1 of \cite{ghosh09} may be applied to give concentration bounds for $Z$. The application of this result relies on writing $Z^\star$ as $Z+X$, for some non-negative random variable $X$ independent of $Z$. Using the representations of (\ref{eq:borelrep}) and Remark \ref{rem:aldous}, we have that $X=Z_{\xi+1}+\cdots+Z_{\eta}-1$, where we note that $\xi+1$ is stochastically smaller than $\eta$ by Theorem 1.1(g) of \cite{klenke10}. Noting that $\mathbb{E}X=(\mathbb{E}\eta-\mathbb{E}\xi)\mathbb{E}Z-1=\frac{\lambda}{(1-\lambda)^2}$, Theorem 5.1 of \cite{ghosh09} gives the lower tail bound
\begin{equation*}
\mathbb{P}\left(\frac{Z-\mathbb{E}[Z]}{\sqrt{\mbox{Var}(Z)}}\leq -t\right)\leq\exp\left\{-\frac{t^2}{2}\right\} \,,
\end{equation*}
for all $t>0$. We could also use Theorem 5.1 of \cite{ghosh09} to give an upper bound for the right tail of $Z$, though this requires finding an upper bound for $\mathbb{E}[Xe^{\gamma X}]$, where $\gamma>0$ is such that $\mathbb{E}[e^{\gamma Z}]<\infty$. Given the relatively complicated structure of the random variable $X$ here, we instead give an alternative method for deriving an upper tail bound which uses the representation (\ref{eq:borelsizebias}) as a starting point and follows an argument analogous to that of Theorem 5.1 of \cite{ghosh09}. One advantage of this approach is that instead of bounding an expression of the form $\mathbb{E}[Xe^{\gamma X}]$, we need to bound $\mathbb{E}[Z^\star e^{\gamma Z^\star}]$ for some $\gamma>0$ specified later. This is considerably simpler. We have the following.
\begin{theorem}
Let $Z\sim\mbox{Borel}(\lambda)$. Let $\delta>0$ be such that $\lambda-\log(\lambda)-1-\delta>0$ and define
\[
K=\frac{\lambda}{2}\left(\frac{(1-\lambda)e^{-\delta}}{\lambda\sqrt{2\pi}(1-e^{-\delta})^2}+\frac{1}{(1-\lambda)^2}\right)
\]
and $\gamma=\lambda-\log(\lambda)-1-\delta$. Then, for all $t>0$,
\[
\mathbb{P}\left(\frac{Z-\mathbb{E}[Z]}{\sqrt{\mbox{Var}(Z)}}\geq t\right)\leq\left\{\begin{array}{lr}\exp\left\{-\frac{\lambda t^2}{2K(1-\lambda)^2}\right\},&\mbox{if }t<\frac{K\gamma(1-\lambda)^2}{\lambda}\,,\\
\exp\left\{-\gamma t+\frac{K\gamma^2(1-\lambda)^2}{2\lambda}\right\},&\mbox{otherwise}\,.
\end{array}\right.
\]
\end{theorem}
\begin{proof}
Let $m(\theta)=\mathbb{E}[e^{\theta Z}]$ for $\theta\in\mathbb{R}$. Then, using the well-known fact that $j!\geq\sqrt{2\pi}j^{j+\frac{1}{2}}e^{-j}$ and noting that $\delta>0$,
\begin{multline*}
\mathbb{E}[Z^\star e^{\gamma Z^\star}]=\frac{\mathbb{E}[Z^2 e^{\gamma Z}]}{\mathbb{E}[Z]}=(1-\lambda)\sum_{j=1}^\infty\frac{e^{(\gamma-\lambda)j}\lambda^{j-1}j^{j+1}}{j!}\\
\leq\frac{(1-\lambda)}{\lambda\sqrt{2\pi}}\sum_{j=1}^\infty\sqrt{j}e^{-\delta j}\leq\frac{(1-\lambda)}{\lambda\sqrt{2\pi}}\sum_{j=1}^\infty j e^{-\delta j}=\frac{(1-\lambda)e^{-\delta}}{\lambda\sqrt{2\pi}(1-e^{-\delta})^2}\,.
\end{multline*}
Combining this with (\ref{eq:borelsizebias}) and the inequality $e^y-e^x\leq\frac{1}{2}(y-x)(e^y+e^x)$, we get that, for all $\theta\in(0,\gamma)$,
\begin{align}
\nonumber\mathbb{E}[e^{\theta Z^\star}-e^{\theta Z}]&=\mathbb{E}[e^{\theta[(1-I)Z+I(Z+Z^\star)]}-e^{\theta Z}]=\lambda\mathbb{E}[e^{\theta(Z+Z^\star)}-e^{\theta Z}]\leq\frac{\lambda\theta}{2}\mathbb{E}[Z^\star(e^{\theta Z^\star}+1)e^{\theta Z}]\\
\nonumber&=\frac{\lambda\theta m(\theta)}{2}\left(\mathbb{E}[Z^\star e^{\theta Z^\star}]+\mathbb{E}[Z^\star]\right)\leq\frac{\lambda\theta m(\theta)}{2}\left(\mathbb{E}[Z^\star e^{\gamma Z^\star}]+\mathbb{E}[Z^\star]\right)\\
\label{eq:conc1}&\leq\frac{\lambda\theta m(\theta)}{2}\left(\frac{(1-\lambda)e^{-\delta}}{\lambda\sqrt{2\pi}(1-e^{-\delta})^2}+\frac{\mathbb{E}[Z^2]}{\mathbb{E}[Z]}\right)=K\theta m(\theta).
\end{align}
With a bound of the form (\ref{eq:conc1}), the argument in the proof of Theorem 5.1 of \cite{ghosh09} shows that $m^\prime(\theta)\leq\mathbb{E}[Z](1+K\theta)m(\theta)$ for any $\theta\in(0,\gamma)$, and that, for any $t>0$,
\begin{equation}\label{eq:conc2}
\mathbb{P}\left(\frac{Z-\mathbb{E}[Z]}{\sqrt{\mbox{Var}(Z)}}\geq t\right)\leq \exp\left\{-\theta t+\frac{K\mathbb{E}[Z]\theta^2}{2\mbox{Var}(Z)}\right\}=\exp\left\{-\theta t+\frac{K(1-\lambda)^2\theta^2}{2\lambda}\right\}\,.
\end{equation}
The exponent on the right-hand side here is minimized at $\theta=\frac{\lambda t}{K(1-\lambda)^2}$. We may make this choice of $\theta$ in (\ref{eq:conc2}) when $\frac{\lambda t}{K(1-\lambda)^2}<\gamma$; otherwise we choose $\theta=\gamma$ in the upper bound (\ref{eq:conc2}). This completes the proof.
\end{proof}

\end{document}